\documentclass[a4paper]{amsart}
\usepackage{latexsym,bm,stmaryrd}
\usepackage{amsmath,amsthm,amsfonts,amssymb,mathrsfs,pb-diagram}
\usepackage[a4paper,hmargin=26mm,vmargin=28mm]{geometry}
\usepackage{microtype}
\usepackage{mathtools}
\usepackage{mathbbol,wasysym}

\newcommand\blam{{\boldsymbol\lambda}}

\newcommand\bmu{{\boldsymbol\mu}}

\let\<=\langle
\let\>=\rangle
\def\({\big(}
\def\){\big)}
\def\Z{\mathbb{Z}}
\def\Q{\mathbb{Q}}

\def\N{\mathbb{N}}

\def\lam{\lambda}
\def\Lam{\Lambda}
\def\Sym{\mathfrak{S}}

\newcommand\HH{\mathscr{H}}

\def\P{\mathscr{P}}

\def\wf{\widetilde{F}}

\DeclareMathOperator\Hom{Hom}
\DeclareMathOperator\End{End}

\DeclareMathOperator\Hc{H}

\title[Cyclotomic nilHecke algebras]
{Schubert Class and cyclotomic nilHecke algebras}
\subjclass[2010]{20C08, 16G99, 06B15}
\keywords{Cyclotomic nilHecke algebras, Grassmannian, Schur polynomials}
\author{Kai Zhou}
  \address{School of Mathematical Sciences\\
  Zhejiang University\\
  Hangzhou, 310027, P.R. China}
  \email{11635017@zju.edu.cn}
\author{Jun Hu}\address{School of Mathematical and Statistics\\
  Beijing Institute of Technology\\
  Beijing, 100081, P.R. China}
\email{junhu404@bit.edu.cn}

\numberwithin{equation}{section}
\newtheorem{prop}[equation]{Proposition}
\newtheorem{thm}[equation]{Theorem}
\newtheorem{cor}[equation]{Corollary}

\newtheorem{lem}[equation]{Lemma}

\theoremstyle{definition}
\newtheorem{dfn}[equation]{Definition}
\theoremstyle{remark}
\newtheorem{rem}[equation]{Remark}
\begin{document}

\begin{abstract} Let $\ell, n$ be positive integers such that $\ell\geq n$. Let $\mathbb{G}_{n,\ell}$ be the Grassmannian which consists of the set of $n$-dimensional subspaces of $\mathbb{C}^{\ell}$. There is a $\mathbb{Z}$-graded algebra isomorphism between the cohomology $\Hc^*(\mathbb{G}_{n,\ell},\Z)$ of $\mathbb{G}_{n,\ell}$ and a natural $\Z$-form $B$ of the $\mathbb{Z}$-graded basic algebra of the type $A$ cyclotomic nilHecke algebra $\HH_{\ell,n}^{(0)}=\<\psi_1,\cdots,\psi_{n-1},y_1,\cdots,y_n\>$. In this paper, we show that the isomorphism can be chosen such that the image of each (geometrically defined) Schubert class $(a_1,\cdots,a_{n})$ coincides with the basis element $b_{\blam}$ constructed by Jun Hu and Xinfeng Liang by purely algebraic method, where $0\leq a_1\leq a_2\leq\cdots\leq a_{n}\leq \ell-n$ with $a_i\in\Z$ for each $i$, $\blam$ is the $\ell$-multipartition of $n$ associated to $\(\ell+1-(a_{n}+n), \ell+1-(a_{n-1}+n-1),\cdots,\ell+1-(a_1+1))$.
A similar correspondence between the Schubert class basis of the cohomology of the Grassmannian $\mathbb{G}_{\ell-n,\ell}$ and the $b_{\lam}$'s basis of the natural $\Z$-form $B$ of the $\mathbb{Z}$-graded basic algebra of $\HH_{\ell,n}^{(0)}$ is also obtained. As an application, we obtain a second version of Giambelli formula for Schubert classes.
\end{abstract}

\maketitle
\setcounter{tocdepth}{1}

\section{Introduction}

The nilHecke algebra of type $A$ was introduced by Kostant and Kumar \cite{KoK} and plays an important role in the study of cohomology of flag varieties and Schubert calculus, see \cite{Hi}.
In recent years, these algebras and their cyclotomic quotients have found some remarkable applications in the categorification of quantum groups, see \cite{KK}, \cite{KhovLaud:diagI}, \cite{KLMS}, \cite{KLM},  \cite{Lau}, \cite{Lau2}, \cite{LV}, \cite{Mathas:Singapore} and \cite{Rou0}. First, let us recall their definitions.

\begin{dfn}\label{nilHecke}
Let $\ell,n\in\N$ and $K$ be any field. The nilHecke algebra $\HH_{n}^{(0)}:=\HH_{n}^{(0)}(K)$ of type $A$
is the unital associative $K$-algebra generated by $\psi_1,\cdots,\psi_{n-1},y_1,\cdots,y_n$
which satisfy the following relations: $$\begin{aligned}
& \psi_r^2=0,\quad \forall\,1\leq r<n,\\
& \psi_r\psi_k=\psi_k\psi_r,\quad\forall\, 1\leq k< r-1< n-1,\\
& \psi_r\psi_{r+1}\psi_r=\psi_{r+1}\psi_r\psi_{r+1},\quad\forall\,1\leq r<n-1,\\
& y_r y_k=y_k y_r,\quad\forall\,1\leq r,k\leq n,\\
& \psi_r y_{r+1}=y_r\psi_r+1,\quad y_{r+1}\psi_r=\psi_r y_r+1,\quad\forall\,1\leq r<n,\\
& \psi_ry_k=y_k\psi_r,\quad\forall\, k\neq r,r+1.
\end{aligned}
$$
The cyclotomic nilHecke algebra $\HH_{\ell,n}^{(0)}:=\HH_{\ell,n}^{(0)}(K)$ of type $A$ is the quotient of $\HH_{n}^{(0)}$ by the two-sided ideal generated by $y_1^\ell$.
\end{dfn}

It is clear that both $\HH_{n}^{(0)}$ and $\HH_{\ell,n}^{(0)}$ are $\Z$-graded $K$-algebras such that each $\psi_r$ is homogeneous with $\deg\psi_r=-2$ and each $y_s$ is homogeneous with $\deg y_s=2$ for all $1\leq r<n, 1\leq s\leq n$. Let $*$ be the $K$-algebra anti-involution of $\HH_{n}^{(0)}$ (or $\HH_{\ell,n}^{(0)}$) which is defined on generators by $\psi_r^*=\psi_r, y_k^*=y_k$ for each $1\leq r<n, 1\leq k\leq n$.
By \cite[(2.7)]{HuLiang} we know that $\HH_{\ell,n}^{(0)}=0$ unless $n\leq\ell$. Henceforth we always assume $n\leq\ell$.

In case some readers are not familiar with the definition of basic algebra, we include it here.
\begin{dfn}
Let $A$ be a $K$-algebra. Suppose $\{P_i\mid i\in I\}$, where $I$ is an index set, is a complete set of non-isomorphic
indecomposable projective $A$-modules, we define the endomorphism algebra $B(A):=\End_A(\oplus_{i\in I}P_i)$ to be the basic algebra of $A$.
\end{dfn}

In a recent work \cite{HuLiang}, the second author and Xinfeng Liang constructed a monomial basis of the cyclotomic nilHecke algebra $\HH_{\ell,n}^{(0)}$ and showed that the $\mathbb{Z}$-graded basic algebra $B\bigl(\HH_{\ell,n}^{(0)}\bigr)$ of $\HH_{\ell,n}^{(0)}$ is commutative and is hence isomorphic to its center $Z(\HH_{\ell,n}^{(0)})$. Furthermore, they constructed an integral basis $\{b_\blam\mid\blam\in\P(0)\}$ for the basic algebra $B\bigl(\HH_{\ell,n}^{(0)}\bigr)$ and an integral basis $\{z_{\blam}\mid\blam\in\P(0)\}$ for the center $Z(\HH_{\ell,n}^{(0)})$ by a purely algebraic method, where each $z_\blam$ is the evaluation of certain unknown symmetric polynomial at $\{y_1,\cdots,y_n\}$, and $\P(0)$ is the set of $\ell$-multipartitions of $n$ with each component being either $(1)$ or empty. Note that $\P(0)$ is in a natural bijection with the following set \begin{equation}\label{pLN}
\P_{\ell,n}:=\bigl\{\blam=(k_1,\cdots,k_{n})\bigm|1\leq k_1<k_2<\cdots<k_{n}\leq \ell, k_i\in\Z,\forall\,i\bigr\} .
\end{equation}
Henceforth we shall identify $\P(0)$ with $\P_{\ell,n}$ without further comments.

In \cite{Lau2}, Lauda proved that the cyclotomic nilHecke algebra $\HH_{\ell,n}^{(0)}$ over $\Q$ is isomorphic to the $n!\times n!$ matrix ring over the cohomology ring of the Grassmannian $\mathbb{G}_{n,\ell}$ with coefficient in $\Q$. In particular, there is an isomorphism between the basic algebra $B\bigl(\HH_{\ell,n}^{(0)}\bigr)$ of $\HH_{\ell,n}^{(0)}$ and the cohomology ring $\Hc^*(\mathbb{G}_{n,\ell},\Q)$ of the Grassmannian $\mathbb{G}_{n,\ell}$. Note that both the algebra $\HH_{\ell,n}^{(0)}$ and its basic algebra $B\bigl(\HH_{\ell,n}^{(0)}\bigr)$ are defined over $\Z$, and Lauda's isomorphism is actually well-defined over $\Z$. We define \begin{equation}\label{thetaLN}
\Theta_{\ell,n}:=\bigl\{(a_1,\cdots,a_{n})\bigm|0\leq a_1\leq\cdots\leq a_{n}\leq \ell-n, a_i\in\Z,\forall\,i\bigr\} .
\end{equation}
By \cite[Chapter III, \S3]{Hi}, for each $(a_1,\cdots,a_{n})\in\Theta_{\ell,n}$, there is a \textbf{Schubert class} $(a_1,\cdots,a_n)\in\Hc^*(\mathbb{G}_{n,\ell},\Z)$. Moreover, the elements in the set
$\{(a_1,\cdots,a_{n})\mid(a_1,\cdots,a_n)\in\Theta_{\ell,n}\}$ form an $\Z$-basis of $\Hc^*(\mathbb{G}_{n,\ell},\Z)$. Therefore, it is natural to ask what the preimage of these geometrically defined Schubert class basis elements in the natural $\Z$-form $B$ of the basic algebra of $\HH_{\ell,n}^{(0)}$ are.

The starting point of this work is to answer the above question. Using Borel's picture for the cohomology of the Grassmannian $\mathbb{G}_{n,\ell}$, we construct an explicit $\Z$-algebra isomorphism between a natural $\Z$-form $B$ of the basic algebra of $\HH_{\ell,n}^{(0)}$  and the cohomology ring $\Hc^*(\mathbb{G}_{n,\ell},\Z)$ of the Grassmannian $\mathbb{G}_{n,\ell}$ such that the purely algebraic defined element
$b(\blam)$(denoted by $b_{\blam}$ in \cite{HuLiang}) is sent to the geometrically defined Schubert class basis element $(a_1,\cdots,a_{n})$, where $\blam\in\P(0)$ is the $\ell$-multipartition associated to the $n$-tuple $(\ell+1-(a_n+n), \ell+1-(a_{n-1}+n-1),\cdots, \ell+1-(a_1+1))\in\P_{\ell,n}$. A similar isomorphism between the natural $\Z$-form $B$ of the basic algebra of $\HH_{\ell,n}^{(0)}$ and the cohomology ring $\Hc^*(\mathbb{G}_{\ell-n,\ell},\Z)$ of the Grassmannian $\mathbb{G}_{\ell-n,\ell}$ is constructed too. On the other hand, it is well-known that there exist some non-canonical isomorphisms between the Grassmannian $\mathbb{G}_{n,\ell}$ and the Grassmannian $\mathbb{G}_{\ell-n,\ell}$, see \cite[Exercise 3.2.6]{Man}. Hence $\Hc^*(\mathbb{G}_{n,\ell},\Z)$ and $\Hc^*(\mathbb{G}_{\ell-n,\ell},\Z)$ are isomorphic to each other as $\Z$-algebras. We expect that the idea and method used in the current paper can be applied in the study of the odd nilHecke algebras and their cyclotomic quotients as well as the odd analogues of the cohomology rings of Grassmannian (cf. \cite{EKL}).
\medskip

The content of the paper is organised as follows. In Section 2, we first introduce some definitions and notations. After recalling some basic knowledge about nilHecke algebras, we then identify (up to a sign) the center element $z_{\blam}$ defined in \cite[Definition 3.3]{HuLiang}
with the evaluation of the Schur symmetric polynomials $s_{\rho(\blam)}(x_1,\cdots,x_n)$ at $x_1:=y_1, \cdots, x_n:=y_n$.
In Section 3, we establish the Pieri formula and Jacobi-Trudi formula for the elements $\{b_{\lambda}\mid\lam\in\P_n\}$,
where $\P_n$ is the set of partitions $\lam$ with length $\ell(\lam)\leq n$.
Using the Giambelli formula for the Schubert class basis element, we give the proof of
the main result Proposition \ref{mainthprop2} which constructs an isomorphism
between $\Hc^*(\mathbb{G}_{n,\ell},\Z)$ and the natural $\Z$-form $B$ of the basic algebra of $\HH_{\ell,n}^{(0)}$
sending the Schubert class basis elements to the $b(\blam)$ (denoted by $b_{\blam}$ in \cite{HuLiang}) basis elements.
In Section 4, we first present a similar isomorphism between $\Hc^*(\mathbb{G}_{\ell-n,\ell},\Z)$
and the natural $\Z$-form $B$ of the basic algebra of $\HH_{\ell,n}^{(0)}$ that sending the Schubert class basis elements
to $b(\blam)$ basis elements. Then we use it to characterize explicitly the image of any
Schubert class element $(a_1,\cdots,a_{n})$ under some natural isomorphism between
$\Hc^*(\mathbb{G}_{n,\ell},\Z)$ and $\Hc^*(\mathbb{G}_{\ell-n,\ell},\Z)$.
As an application, we get a second version of Giambelli formula for Schubert classes in Corollary \ref{lastcor}.

\bigskip
\section*{Acknowledgements}

 The research was supported by the National Natural Science Foundation of China (No. 11525102).
\bigskip

\section{Preliminaries}

Let $m\in\N$. A partition $\lam=(\lam_1,\lam_2, \cdots)$ of $m$ is a non-increasing sequence of non-negative integers which sums to $m$. If  $\lam=(\lam_1,\lam_2,\cdots)$ is a partition of $m$, then we write $|\lam|=m$ and say that the size of $\lam$ is $m$. The length $\ell(\lam)$ of $\lam$ is defined to be the largest integer $k$ such that $\lam_k\neq 0$. A partition is uniquely determined by its diagram. The \textbf{conjugate} of $\lam$ is the partition $\lam'$ such that $\lam'_k=\#\{j\mid\lam_j\geq k\}$ for each $k\geq 1$. Let $\P_n:=\{\lam=(\lam_1,\cdots,\lam_n)\mid\lam_1\geq\lam_2\geq\cdots\geq\lam_n,\lam_i\in\N,\forall\,i\}$, i.e., the set of partitions $\lam$ with $\ell(\lam)\leq n$.

Let $\ell,n\in\N$ with $\ell\geq n$. Let $\P(0)$ be the set of $\ell$-multipartitions of $n$ with each component being either $(1)$ or empty. That is, $$ \P(0):=\biggl\{\blam={(\lambda^{(1)},\cdots,\lambda^{(\ell)})}\biggm|\sum_{j=1}^{\ell}|\lambda^{(j)}|=n,\lambda^{(j)}\in\{\emptyset,(1)\},\forall\,j\biggr\} .
$$
There is a natural bijection $\theta$ between $\P(0)$ and $\P_{\ell,n}$ which is defined as follows: for $\blam=(\lambda^{(1)},\dots,\lambda^{(\ell)})\in\P(0)$, we define $\theta(\blam)$ to be the unique $n$-tuple $(k_1,\dots,k_n)\in\P_{\ell,n}$ such that for each $1\leq j\leq\ell$, $$\lambda^{(j)}=\begin{cases}
    (1), & \text{ if }\ j=k_i\ \text{for some}\ 1\leq i\leq n,\\
    \emptyset, & \text{ otherwise }.
\end{cases}$$
{\it Henceforth we shall identify $\P(0)$ with $\P_{\ell,n}$ using the above bijection.}

\begin{dfn}\label{injective}
We define an injective map $\rho: \P_{\ell,n}\hookrightarrow\P_n$ as follows: for any $\blam=(k_1,\cdots,k_n)\in\P_{\ell,n}$, set $$\rho(\blam):=(\ell-k_1-n+1,\ell-k_2-n+2,\dots,\ell-k_n).$$
\end{dfn}

Let $\Sym_n$ be the symmetric group on $\{1,2,\dots,n\}$. For each $1\leq i<n$, set $s_i:=(i,i+1)\in\Sym_n$. Then $\{s_1,\cdots,s_{n-1}\}$ is the standard set of Coxeter generators for $\Sym_n$. If $w\in\Sym_n$ then the length of $w$ is $$
\ell(w):=\min\{k\in\N\mid\text{$w=s_{i_1}\cdots s_{i_k}$ for some $1\leq i_1,\cdots,i_k<n$}\}.
$$
Let $w_{0}$ be the unique longest element in $\Sym_n$.

Let $\HH_{n}^{(0)}=\<\psi_r,y_k\mid1\leq r<n,1\leq k\leq n\>$ and $\HH_{\ell,n}^{(0)}=\<\psi_r,y_k\mid1\leq r<n,1\leq k\leq n\>$ be the type $A$ nilHecke algebra and the type $A$ cyclotomic nilHecke algebra as introduced in Definition \ref{nilHecke} respectively. Let $$\pi_\ell: \HH_{n}^{(0)}\twoheadrightarrow\HH_{\ell,n}^{(0)}$$ be the natural surjection which sends $\psi_r$ to $\psi_r$ and $y_k$ to $y_k$ for each $1\leq r<n, 1\leq k\leq n$. Suppose $w\in\Sym_n$. If $w=s_{i_1}\cdots s_{i_k}$ with $k=\ell(w)$ then $s_{i_1}\cdots s_{i_k}$ is a reduced expression for $w$.
In this case, we define $\psi_w:=\psi_{i_1}\cdots\psi_{i_k}$ in $\HH_n^{(0)}$(or in $\HH_{\ell,n}^{(0)}$).
The braid relations in Definition \ref{nilHecke} ensures that $\psi_w$ does not depend on the choice of the reduced expression for $w$.

Let $K[x_1,\cdots,x_n]$ be the polynomial ring with $n$ indeterminates $x_1,\cdots,x_n$.
We define the action of $\Sym_n$ on $K[x_1,\dots,x_n]$ by
$$w\cdot f(x_1,\dots,x_n):=f(x_{w(1)},\dots,x_{w(n)}),\qquad\forall w\in\Sym_n,\,\forall f\in K[x_1,\dots,x_n].$$
Following \cite{Rou0}, for each $1\leq i<n$, we define $\partial_i$ to be the $K$-linear maps on $K[x_1,\dots,x_n]$ such that
$$\partial_i(f)=\cfrac{f-s_i\cdot f}{x_{i+1}-x_i},\qquad \forall\,f\in K[x_1,\dots,x_n].$$

\begin{lem}[\cite{Rou0}]\label{the relations of Demazure operators}
We have
\begin{align*}
\partial_i^2&=0,\qquad 1\leq i<n,\\
\partial_{i+1}\partial_i\partial_{i+1}&=\partial_i\partial_{i+1}\partial_i,\qquad 1\leq i<n-1,\\
\partial_{i}\partial_{j}&=\partial_j\partial_{i},\qquad 1\leq i<j-1<n-1.
\end{align*}
\end{lem}

Let $w\in\Sym_n$. If $w=s_{i_1}\cdots s_{i_k}$ is a reduced expression for $w$,
then we define the $K$-linear map $\partial_w:=\partial_{i_1}\cdots\partial_{i_k}$.
The braid relations in Lemma \ref{the relations of Demazure operators}
ensures that $\partial_w$ does not depend on the choice of the reduced expression for $w$.

It is well known that $K[x_1,\dots,x_n]$ becomes a faithful representation of $\HH_{n}^{(0)}$ as follows: $\forall f\in K[x_1,\dots,x_n]$,
$$\psi_i\cdot f=\partial_i(f),\quad y_k\cdot f=x_kf,\qquad \forall 1\leq i<n,\,\forall 1\leq k\leq n.$$
For each $w\in\Sym_n$, it is clear that $\psi_w$ acts on $K[x_1,\dots,x_n]$ by the $K$-linear map $\partial_{w}$.

\begin{lem}[\cite{KhovLaud:diagI}] \label{standardbasis}
The elements in the set $\{\psi_wy_1^{c_1}\cdots y_n^{c_n}\mid w\in\Sym_n, c_1,\cdots,c_n\in\N\}$
form a $K$-basis of the nilHecke algebra $\HH_{n}^{(0)}$. Similarly, the elements in the set
$\{y_1^{c_1}\cdots y_n^{c_n}\psi_w\mid w\in\Sym_n, c_1,\cdots,c_n\in\N\}$ form a $K$-basis of the nilHecke algebra $\HH_{n}^{(0)}$.
\end{lem}

According to Lemma \ref{standardbasis}, it is clear that $x_i\mapsto y_i (1\leq i\leq n)$
induce a $K$-algebra isomorphism between $K[x_1,\dots,x_n]$ and the subalgebra of $\HH_{n}^{(0)}$ generated by $y_1,\dots,y_n$.
Henceforth, $y_1,\dots,y_n$ are algebraically {\it independent} in $\HH_n^{(0)}$. So we can use $K[y_1,\dots,y_n]$ to denote
the subalgebra of $\HH_{n}^{(0)}$ generated by $y_1,\dots,y_n$.
We warn the readers that though we use the same symbols to denote both the generators in $\HH_{n}^{(0)}$ and in $\HH_{\ell,n}^{(0)}$ by some abuse of notations, it should be clear that they are elements in different algebras.
For example, $y_1,\cdots,y_n$ are algebraically {\it dependent} in $\HH_{\ell,n}^{(0)}$.

By the defining relations in Definition \ref{nilHecke}, for each $1\leq s<n$, one can easily check that
$$\psi_s(y_s+y_{s+1})=(y_s+y_{s+1})\psi_s,\quad \psi_s(y_sy_{s+1})=(y_sy_{s+1})\psi_s,\quad \psi_s y_k=y_k\psi_s,\quad\forall\,k\notin\{s,s+1\} .$$
As a consequence, for any $r,k\in\Z^{\geq 1}$ with $r+k\leq n$,
the symmetric polynomials in $y_r, y_{r+1}, \cdots, y_{r+k}$ commute with the elements $\psi_s$ with $s\in\N$ and $r\leq s<r+k$.
Furthermore, the center $Z(\HH_n^{(0)})$ of $\HH_{n}^{(0)}$ is the set of symmetric polynomials in $y_1,\cdots,y_n$.

Let $\Lambda_n$ be the symmetric polynomial ring with $n$ variables $x_1,\cdots,x_n$.
For each $0\leq k\leq n$, let $e_k\in\Lam_n$ be the $k$th elementary symmetric polynomials.
For each $s\geq 0$, let $h_s\in\Lam_n$ be the $s$th complete symmetric polynomials.
By some abuse of notation, we still denote by $\Lambda_n$ the evaluation of the symmetric polynomial
ring at $x_1:=y_1,\dots,x_n:=y_n$, that is the center $Z(\HH_n^{(0)})$ of $\HH_{n}^{(0)}$.

\begin{dfn}\label{SlamDef}
For each partition $\lambda=(\lambda_1,\dots,\lambda_n)\in\P_{n}$, we define an element inside $\HH_{n}^{(0)}$ as follows:
$$S_{\lambda}:=(-1)^{\tfrac{n(n-1)}{2}}\psi_{w_{0}}y_1^{\lambda_1+n-1}y_2^{\lambda_2+n-2}\cdots y_n^{\lambda_n}.$$
For a partition $\mu$ with length $\ell(\mu)>n$, we define $S_{\mu}:=0$, in particular, $S_{(1^s)}=0$ if $s>n$. We also
define $S_{(s)}:=0$, $S_{(1^{s})}:=0$ if $s<0$.
\end{dfn}

We define $J_n$ to be the left ideal of $\HH_{n}^{(0)}$ generated by $\psi_1,\dots,\psi_{n-1}$
and $J_n^*$ to be the image of $J_n$ under the anti-involution $*$. That is,
\begin{equation}\label{2Jn}
J_n=\sum_{s=1}^{n-1}\HH_{n}^{(0)}\psi_{s},\quad J^*_n=\sum_{s=1}^{n-1}\psi_s\HH_{n}^{(0)}.
\end{equation}

\begin{lem}\label{decompositions as K-linear space}
We have $K$-linear space decompositions,
$$\HH_{n}^{(0)}=K[y_1,\dots,y_n]\oplus J_n,\quad \HH_{n}^{(0)}=K[y_1,\dots,y_n]\oplus J^*_n$$
\end{lem}
\begin{proof}
It is clear that $\HH_{n}^{(0)}=K[y_1,\dots,y_n]+J_n$. According to the defining relations in Definition \ref{nilHecke} and the definition of $\psi_w$,
for any $1\leq s<n$ and $w\in\Sym_n$,
\begin{align}\label{usefull formula}
\psi_{s}\psi_w=\begin{cases}
\psi_{sw},&\text{ if }\ell(sw)=\ell(w)+1,\\
0,&\text{ if }\ell(sw)=\ell(w)-1,
\end{cases}\quad
\psi_{w}\psi_s=\begin{cases}
\psi_{ws},&\text{ if }\ell(ws)=\ell(w)+1,\\
0,&\text{ if }\ell(ws)=\ell(w)-1.
\end{cases}
\end{align}
Therefore, according to Lemma \ref{standardbasis}, elements in $J_n$ can be expressed as the $K$-linear sums of
elements $y_1^{c_1}\cdots y_n^{c_n}\psi_w$ with $w\neq 1$ and $c_1,\cdots,c_n\in\N$.
Hence $K[y_1,\dots,y_n]\cap J_n=\{0\}$ and  $\HH_{n}^{(0)}=K[y_1,\dots,y_n]\oplus J_n$.
For the second decomposition, we just need to apply $*$ to the first decomposition.
\end{proof}

As a consequence, for each $a\in \HH_{n}^{(0)}$, there is a unique
element $f_a\in K[y_1,\cdots,y_n]$ and a unique element $g_a\in K[y_1,\cdots,y_n]$ such that
$$a\equiv f_a\mod J_n,\qquad a\equiv g_a\mod J^*_n.$$
We call $f_a$ the polynomial part of $a$. Applying $*$ to the first equivalence,
by the uniqueness of $g_{a^*}$, we derive $f_a=g_{a^*}$.

Hence we have the following definition.
\begin{dfn}\label{zlamDef}
Let $\lam=(\lambda_1,\dots,\lambda_n)\in\P_{n}$. Inside $\HH_{n}^{(0)}$,
we define $z_{\lam}$ to be the unique polynomial in $K[y_1,\cdots,y_n]$ such that
$$S_{\lam}\equiv z_{\lam}\mod J_n.$$
\end{dfn}

\begin{dfn}[\cite{Man}]\label{Schur polynomial}
For each partition $\lambda=(\lam_1,\dots,\lam_n)\in\P_n$, set
$$a_{\lambda}(x_1,\cdots,x_n):=\sum_{w\in\Sym_n}(-1)^{\ell(w)}x_{w(1)}^{\lambda_1}\cdots x_{w(n)}^{\lam_n},$$
the Schur polynomial associated to $\lambda$ is $s_{\lambda}(x_1,\cdots,x_n):=\tfrac{a_{\lambda+\delta}(x_1,\cdots,x_n)}{a_{\delta}(x_1,\cdots,x_n)}$
where $\delta=(n-1,n-2,\dots,1,0)$.
\end{dfn}

It is clear that the elementary symmetric polynomials $e_k\in\Lam_n(0\leq k\leq n)$
and the complete symmetric polynomials $h_t\in\Lam_n(t\geq 0)$ are special Schur polynomials by
$$e_{k}(x_1,\dots,x_{n})=s_{(1^k)}(x_1,\dots,x_n),\quad h_{t}(x_1,\dots,x_n)=s_{(t)}(x_1,\dots,x_n).$$
For each $\lam=(\lam_1,\dots,\lam_n)\in\P_n$, it is well known that
$$(-1)^{\frac{n(n-1)}{2}}\partial_{w_0}(x_1^{\lam_1+n-1}x_2^{\lam_2+n-2}\cdots x_{n}^{\lam_n})=s_{\lam}(x_1,\dots,x_n),$$
refer to \cite[(2.57)]{KLMS}. We remark readers that the $K$-linear map $D_n$ defined in \cite{KLMS}
is equal to $(-1)^{\frac{n(n-1)}{2}}\partial_{w_0}$ in our notations.
For more explicit definition and properties of Schur polynomials, we refer the readers to Chapter 1 of Manivel's book \cite{Man}.

\begin{cor}\label{the equivalence between zlam and Schur polynomial}
Let $1$ be the identity element of $K[x_1,\dots,x_n]$. For each $\lam=(\lam_1,\dots,\lam_n)\in\P_n$,
we have $$S_{\lam}\cdot 1=s_{\lam}(x_1,\dots,x_n).$$
Hence $z_{\lam}$ is the evaluation of $s_{\lam}(x_1,\dots,x_n)$ at $x_1:=y_1,\dots,x_n:=y_n$,
that is,
$$(-1)^{\tfrac{n(n-1)}{2}}\psi_{w_{0}}y_1^{\lambda_1+n-1}y_2^{\lambda_2+n-2}\cdots y_n^{\lambda_n}\equiv s_{\lam}(y_1,\dots,y_n)\mod J_n$$
and
$$(-1)^{\tfrac{n(n-1)}{2}}y_1^{\lambda_1+n-1}y_2^{\lambda_2+n-2}\cdots y_n^{\lambda_n}\psi_{w_0}\equiv s_{\lam}(y_1,\dots,y_n)\mod J^*_n.$$
\end{cor}
\begin{proof}
Observing that $\psi_s\cdot 1=0,\,\forall 1\leq s<n$, we derive
$$z_{\lam}\cdot 1=S_\lam\cdot 1=(-1)^{\frac{n(n-1)}{2}}\partial_{w_0}(x_1^{\lam_1+n-1}x_2^{\lam_2+n-2}\cdots x_{n}^{\lam_n})=s_{\lam}(x_1,\dots,x_n).$$
Hence $s_{\lam}(y_1,\dots,y_n)=z_{\lam}$.
\end{proof}

According to Definition \ref{injective}, Definition \ref{SlamDef} and Definition \ref{zlamDef}, for each $\blam=(k_1,\dots,k_n)\in\P_{\ell,n}$,
$$S_{\rho(\blam)}=(-1)^{\frac{n(n-1)}{2}}\psi_{w_0}y_1^{\ell-k_1}y_2^{\ell-k_2}\cdots y_{n}^{\ell-k_n}\equiv z_{\rho(\blam)}\mod J_n.$$
Applying $*$ to both sides of the equivalence, we obtain
$$(-1)^{\frac{n(n-1)}{2}}y_1^{\ell-k_1}y_2^{\ell-k_2}\cdots y_{n}^{\ell-k_n}\psi_{w_0}\equiv z_{\rho(\blam)}\mod J^*_n.$$
We set $z(\blam):=\pi_\ell\bigl(z_{\rho(\blam)}\bigr)\in\HH_{\ell,n}^{(0)}$.
Hence the elements $z(\blam)$ in our notations here are equal to $(-1)^{\tfrac{n(n-1)}{2}}z_{\blam}$
defined in \cite[Definition 3.3]{HuLiang}.

\begin{thm}(\cite[Theorem 3.7]{HuLiang})\label{centerbasis}
The elements in the set $\{z(\blam)\mid\blam\in\P_{\ell,n}\}$ form a $K$-basis of the center $Z(\HH_{\ell,n}^{(0)})$ of $\HH_{\ell,n}^{(0)}$.
In particular, the center of $\HH_{\ell,n}^{(0)}$ is the set of symmetric polynomials in $y_1,\dots,y_n$.
\end{thm}

Combining Theorem \ref{centerbasis} and Corollary \ref{the equivalence between zlam and Schur polynomial}, we get that

\begin{cor}\label{maincor1}
For each $\blam\in\P_{\ell,n}$, inside $\HH_{\ell,n}^{(0)}$, $z(\blam)=s_{\rho(\blam)}(y_1,\cdots,y_n)$. In particular, the elements in the set $\{s_{\rho(\blam)}(y_1,\cdots,y_n)\mid\blam\in\P_{\ell,n}\}$ form a $K$-basis of the center $Z(\HH_{\ell,n}^{(0)})$ of $\HH_{\ell,n}^{(0)}$.
\end{cor}

\bigskip

\section{The Schubert class basis $\{(a_1,\cdots,a_{n})\}$ versus the basis $\{b(\blam)\}$}

In this section, we shall give the proof of the main result Proposition \ref{mainthprop2}.
The proof relies on both the Jacobi-Trudi formula for the elements $\{b_{\lambda}\mid\lam\in\P_n\}$
and the Giambelli formula for the Schubert class basis elements $\{(a_1,\cdots,a_{n})\mid(a_1,\dots,a_n)\in\Theta_{\ell,n}\}$.
\medskip

Recall that $\mathbb{G}_{n,\ell}$ is the complex Grassmann manifold which consists of the set of $n$-dimensional subspaces of $\mathbb{C}^{\ell}$.
Let $0=V_0\subset V_1\subset V_2\subset\cdots\subset V_{\ell}=\mathbb{C}^{\ell}$ be a fixed complete flag in $\mathbb{C}^\ell$, i.e., each $V_i$ is an $i$-dimensional $\mathbb{C}$-linear subspace of $\mathbb{C}^\ell$.

Recall that by (\ref{thetaLN}), $$
\Theta_{\ell,n}:=\bigl\{(a_1,\cdots,a_{n})\bigm|0\leq a_1\leq\cdots\leq a_{n}\leq \ell-n, a_i\in\Z,\forall\,i\bigr\} .
$$
For each $n$-tuple $(a_1,\dots,a_{n})\in\Theta_{\ell,n}$, we define $$
\<a_1,\cdots,a_{n}\>:=\bigl\{X\in\mathbb{G}_{n,\ell}\bigm|\dim(X\cap V_{a_i+i})=i,\,\dim(X\cap V_{a_i+i-1})=i-1,\,\,\forall\,1\leq i\leq n\bigr\} .
$$
Let $[a_1,\cdots,a_{n}]:=\overline{\<a_1,\cdots,a_{n}\>}$ be the closure of $\<a_1,\cdots,a_{n}\>$.
Then it is well-known that $$
[a_1,\cdots,a_{n}]=\bigl\{X\in\mathbb{G}_{n,\ell}\bigm|\dim(X\cap V_{a_i+i})\geq i,\,\,\forall\,1\leq i\leq n\bigr\}.
$$

\begin{lem}\text{(\cite[(3.1)]{Hi})}
The integral homology $\Hc_*(\mathbb{G}_{n,\ell},\Z)$ is a free $\Z$-module of finite rank with a $\Z$-basis given by
$\{[a_1,\cdots,a_{n}]\mid(a_1,\dots,a_{n})\in\Theta_{\ell,n}\}$.
\end{lem}

Recall that $$
\Hc^{2\sum_{j=1}^{n}a_j}(\mathbb{G}_{n,\ell},\Z)\cong\Hom_{\Z}\bigl(\Hc_{2\sum_{j=1}^{n}a_j}(\mathbb{G}_{n,\ell},\Z),\Z\bigr) .
$$
By some abuse of notation, we define $(a_1,\cdots,a_{n})\in \Hc^{2\sum_{j=1}^{n}a_j}(\mathbb{G}_{n,\ell},\Z)$ by assigning $1$ to $[a_1,\cdots,a_{n}]$
and $0$ to any other $[b_1,\cdots,b_{n}]\neq [a_1,\cdots,a_{n}]$.
There are no confusions in $(a_1,\dots,a_n)$ by the context.

\begin{lem}\label{basis2} The elements in the following set $$
\bigl\{(a_1,\cdots,a_n)\bigm|(a_1,\cdots,a_n)\in\Theta_{\ell,n}\bigr\}
$$
form a $\Z$-basis of $\Hc^*(\mathbb{G}_{n,\ell},\Z)$.
\end{lem}
We call each $(a_1,\cdots,a_{n})$ a Schubert class basis element in $\Hc^{2\sum_{j=1}^{n}a_j}(\mathbb{G}_{n,\ell},\Z)$. Following \cite{Hi}, we know that the Chern classes and the normal Chern classes can be identified as the special Schubert classes as follows: $$\begin{aligned}
&c_i:=(0,\dots,0,\underbrace{1,\dots,1}_{\text{$i$ copies}})\in\Hc^*(\mathbb{G}_{n,\ell},\Z),\quad \forall\,1\leq i\leq n,\\
&\bar{c}_j:=(-1)^j\tilde{c}_j\in\Hc^*(\mathbb{G}_{n,\ell},\Z),\text{ where }\tilde{c}_j:=(0,\dots,0,j)\in\Hc^*(\mathbb{G}_{n,\ell},\Z),\quad
\forall\,1\leq j\leq \ell-n.
\end{aligned}
 $$

\begin{thm}\text{(Borel, \cite{Bor})}\label{borel} Let $x_1,\dots,x_{n},\bar{x}_1,\dots,\bar{x}_{\ell-n}$ be $\ell$-indeterminates over $\Z$. The map $\iota$ which sends $c_i$ to $x_i+I_{n,\ell-n}$ for each $1\leq i\leq n$, and $\bar{c}_j$ to $\bar{x}_j+I_{n,\ell-n}$ for each $1\leq j\leq \ell-n$, can be extended uniquely to an $\Z$-algebra isomorphism $\iota: \Hc^*(\mathbb{G}_{n,\ell},\Z)\cong \mathbb{Z}[x_1,\dots,x_{n},\bar{x}_1,\dots,\bar{x}_{\ell-n}]/I_{n,\ell-n}$,
where $I_{n,\ell-n}$ is the ideal generated by the coefficients of the following equation
$$(1+x_1t+\dots+x_{n}t^{n})(1+\bar{x}_1t+\dots+\bar{x}_{\ell-n}t^{\ell-n})=1.$$
\end{thm}

Henceforth we shall use the isomorphism $\iota$ to identify the cohomology algebra $\Hc^*(\mathbb{G}_{n,\ell},\Z)$ with the quotient algebra $\mathbb{Z}[x_1,\dots,x_{n},\bar{x}_1,\dots,\bar{x}_{\ell-n}]/I_{n,\ell-n}$.

In \cite[Proposition 5.3]{Lau2} Lauda proved  that $\Hc^*(\mathbb{G}_{n,\ell},\mathbb{Q})$ is isomorphic to the basic algebra of $\HH_{\ell,n}^{(0)}$ over $\Q$. Note that the Grassmannian $\mathbb{G}_{\ell-n,\ell}$ is (non-canonically) isomorphic to the Grassmannian $\mathbb{G}_{n,\ell}$ by \cite[Exercise 3.2.6]{Man}, hence $\Hc^*(\mathbb{G}_{\ell-n,\ell},\mathbb{Z})$ is isomorphic to $\Hc^*(\mathbb{G}_{n,\ell},\mathbb{Z})$ and the basic algebra of $\HH_{\ell,\ell-n}^{(0)}$ is isomorphic to the basic algebra of $\HH_{\ell,n}^{(0)}$.

Furthermore, by the graded cellular basis (cf. \cite{HuMathas:GradedCellular}, \cite[(2.7)]{HuLiang}) for the cyclotomic nilHecke algebra $\HH_{\ell,n}^{(0)}$, we know that $\HH_{\ell,n}^{(0)}$ is actually defined over $\Z$. Moreover, Lauda's isomorphism is well-defined over $\Z$ too.

\begin{dfn}\label{bijective} Let $\tau$ be the bijection between $\Theta_{\ell,n}$ and $\P_{\ell,n}$ which is defined as follows:
for any $n$-tuple $\underline{a}=(a_1,\dots,a_{n})\in\Theta_{\ell,n}$, \begin{equation}\label{taudfn}
\tau{(\underline{a})}:=(\ell+1-(a_{n}+n),\ell+1-(a_{n-1}+n-1),\cdots,\ell+1-(a_1+1))\in\P_{\ell,n} . \end{equation}
\end{dfn}

In other words, if $\mu=(\mu^{(1)},\cdots,\mu^{(\ell)})\in\P(0)$ is the $\ell$-multipartition which is associated to $(a_1+1,a_2+2,\cdots,a_n+n)\in\P_{\ell,n}$, then $\mu'=(\mu^{(\ell)},\cdots,\mu^{(1)})\in\P(0)$ is the $\ell$-multipartition which is associated to $\tau(\underline{a})$.

\begin{dfn} \label{blamDFN} For each partition $\lambda=(\lambda_1,\cdots,\lambda_n)\in\P_n$, we define elements inside $\HH_{n}^{(0)}$ as follows:
$$b_{\lambda}:=\psi_{w_{0}}y_1^{\lambda_1+n-1}y_2^{\lambda_2+n-2}\cdots y_n^{\lambda_n}\psi_{w_{0}}y_1^{n-1}y_2^{n-2}\cdots y_{n-1}.$$
For any partition $\mu$ with length $\ell(\mu)>n$, we define $b_{\mu}:=0$, in particular, $b_{(1^s)}=0$ if $s>n$.
We also define $b_{(s)}:=0$, $b_{(1^{s})}:=0$ if $s<0$. For each $\blam=(k_1,\cdots,k_n)\in\P_{\ell,n}$,
we define $b(\blam):=\pi_{\ell}\bigl(b_{\rho(\blam)}\bigr)\in\HH_{\ell,n}^{(0)}$.
\end{dfn}

Note that the above definition is a slight generalization of (\cite[Definitions 3.1]{HuLiang}) in the following sense: if $\blam=(\lam_1,\cdots,\lam_n)\in\P_{\ell,n}$, then the elements $b(\blam)$ in our notations here are the same as the elements
$b_{\blam}$ defined in \cite[Definitions 3.1]{HuLiang}.

For each $\lam=(\lam_1,\dots,\lam_n)\in\P_n$, for simplicity, we define elements inside $\HH_{n}^{(0)}$ or inside $\HH_{\ell,n}^{(0)}$ as follows:
$$y_{\lambda}:=y_1^{\lambda_1+n-1}y_2^{\lambda_2+n-2}\cdots y_n^{\lambda_n},\quad y_{\min}:=y_1^{n-1}y_2^{n-2}\cdots y_{n-1}.$$

By \cite{HuLiang} or \cite{Lau}, we know $\wf_{1,1}\HH_{\ell,n}^{(0)}$ is a
representation element of the unique isomorphic class of indecomposable projective $\HH_{\ell,n}^{(0)}$-module,
where $\wf_{1,1}=(-1)^{\tfrac{n(n-1)}{2}}\psi_{w_{0}}y_1^{n-1}y_2^{n-2}\cdots y_{n-1}$ is a primitive idempotent.
Hence $B(\HH_{\ell,n}^{(0)})=\End_{\HH_{\ell,n}^{(0)}}(\wf_{1,1}\HH_{\ell,n}^{(0)})\cong\wf_{1,1}\HH_{\ell,n}^{(0)}\wf_{1,1}$.

\begin{lem}\label{the integral basis of basic algebra}\text{(\cite[Lemma 3.2]{HuLiang})}
The elements in $\{b(\blam)\mid\blam\in\P_{\ell,n}\}$ form a $K$-basis of the basic algebra $B\bigl(\HH_{\ell,n}^{(0)}\bigr)$ of $\HH_{\ell,n}^{(0)}$.
Moreover, the basic algebra $B\bigl(\HH_{\ell,n}^{(0)}\bigr)$ of $\HH_{\ell,n}^{(0)}$ is commutative and isomorphic to the center $Z(\HH_{\ell,n}^{(0)})$
of $\HH_{\ell,n}^{(0)}$.
\end{lem}

\begin{dfn}
We define $B$ to be the free $\Z$-submodule of the basic algebra of the cyclotomic nilHecke algebra $\HH_{\ell,n}^{(0)}(\Q)$ over $\Q$ generated by $\{b(\blam)\mid\blam\in\P_{\ell,n}\}$. We call $B$ the natural $\Z$-form of $B\bigl(\HH_{\ell,n}^{(0)}\bigr)$.
\end{dfn}

The validity of Lemma \ref{the integral basis of basic algebra} over any field $K$ already implies that
$B$ is a $\Z$-algebra and $B\bigl(\HH_{\ell,n}^{(0)}(K)\bigr)\cong B\otimes_{\Z}K$ for any field $K$.

\begin{dfn} For each $1\leq i\leq n$, $1\leq j\leq \ell-n$, we define $$\begin{aligned}
&\blam_i:=\tau\bigl(\underbrace{0,\dots,0,}_{\text{$n-i$ copies}}\underbrace{1,\dots,1}_{\text{$i$ copies}}\bigr)=(\underbrace{\ell-n,\ell-n+1,\cdots,\ell-n+i-1,}\underbrace{\ell-n+i+1,\ell-n+i+2,\cdots,\ell})\in\P_{\ell,n},\\
&\bmu_j:=\tau\bigl(\underbrace{0,\dots,0,}_{\text{$n-1$ copies}}j\bigr)=(\ell+1-n-j,\underbrace{\ell-n+2,\dots,\ell-1,\ell})\in\P_{\ell,n} .
\end{aligned}
$$
\end{dfn}

By Definition \ref{injective}, $\rho(\blam_i)=(1^i)$ and $\rho(\bmu_j)=(j)$, hence
$$b(\blam_i)=\psi_{w_{0}}(y_1y_2\cdots y_{i})y_{\min}\psi_{w_{0}}y_{\min},\quad\forall\, 1\leq i\leq n, $$ and
$$b(\bmu_j)=\psi_{w_{0}}y_1^{j}y_{\min}\psi_{w_{0}}y_{\min},\quad\forall\, 1\leq j\leq \ell-n.$$

For the convenience of our arguments in the proof of Proposition \ref{mainthprop2} and Theorem \ref{mainthm2}, we set
$$b(\blam_0)=b(\bmu_0):=\psi_{w_0}y_{\min}\psi_{w_0}y_{\min}$$
and
$$b(\blam_i):=0,\quad b(\bmu_j):=0,\qquad\forall i<0\text{ or }i>n,\,\forall j<0\text{ or }j>\ell-n.$$

\begin{lem}[\cite{Mac}]\label{alternating}
Inside $\Lambda_n$, we have that
$$\begin{aligned}&\sum_{s=0}^{n}(-1)^se_sh_{m-s}=0,\quad \text{ if } m\geq n,\\
&\sum_{s=0}^{m}(-1)^se_sh_{m-s}=0,\quad \text{ if } m\leq n.\end{aligned}$$
\end{lem}

\begin{proof} This follows directly from \cite[(2.6), (2.6$'$)]{Mac}.
\end{proof}

\begin{lem}\text{(\cite[Proposition 7]{HoffnungLauda:KLRnilpotency})}\label{sum1}
For any integers $1\leq s\leq n$, $t\geq 1$, inside the cyclotomic nilHecke algebra $\HH_{\ell,n}^{(0)}$, we have that
$$ \sum_{l_1+\cdots+l_s=\ell-s+t}y_1^{l_1}\cdots y_s^{l_s}=0.$$
\end{lem}

The following result generalizes \cite[Lemma 3.2]{HuLiang}. Though the proof is similar, we include the proof here for the convenience of
readers.

\begin{lem}\label{commutation0}
For any $\lambda,\mu\in\P_n$, we have that $b_{\lambda}b_{\mu}=b_{\mu}b_{\lam}$ holds inside $\HH_{n}^{(0)}$
and hence holds for their images inside $\HH_{\ell,n}^{(0)}$.
\end{lem}
\begin{proof}
According to $(\ref{usefull formula})$, one can easily check
that $J_n\psi_{w_0}=0=\psi_{w_0}J_n^*$.
We have that
\begin{align*}
b_{\mu}b_{\lambda}&=\psi_{w_{0}}y_{\mu}\psi_{w_{0}}y_{\min}\psi_{w_{0}}y_{\lambda}\psi_{w_{0}}y_{\min}\\
&=(-1)^{\tfrac{n(n-1)}{2}}\psi_{w_{0}}(y_{\mu}\psi_{w_{0}}y_{\lambda})\psi_{w_{0}}y_{\min}
\end{align*}
and
\begin{align*}
b_{\lambda}b_{\mu}&=\psi_{w_{0}}y_{\lambda}\psi_{w_{0}}y_{\min}\psi_{w_{0}}y_{\mu}\psi_{w_{0}}y_{\min}\\
&=(-1)^{\tfrac{n(n-1)}{2}}\psi_{w_{0}}(y_{\lambda}\psi_{w_{0}}y_{\mu})\psi_{w_{0}}y_{\min}
\end{align*}
We can write
$$y_{\mu}\psi_{w_{0}}y_{\lambda}\equiv h\mod J_n$$ where $h\in K[y_1,\dots,y_n]$. Applying the anti-involution $*$
to the above equivalence, we get that $$y_{\lambda}\psi_{w_{0}}y_{\mu}\equiv h\mod J^*_n.$$
Therefore $$b_{\lambda}b_{\mu}=(-1)^{\tfrac{n(n-1)}{2}}\psi_{w_{0}}h\psi_{w_{0}}y_{\min}=b_{\mu}b_{\lambda}.$$
\end{proof}

\begin{lem}\label{homomorphism} Let $B$ be the natural $\Z$-form of the $\Z$-graded basic algebra of $\HH_{\ell,n}^{(0)}$. The map $\eta$ which sends $c_i=x_i+I_{n,\ell-n}$ to $b(\blam_i)$ for each $1\leq i\leq n$, and $\bar{c}_j=\bar{x}_j+I_{n,\ell-n}$ to $(-1)^jb(\bmu_j)$ for each $1\leq j\leq \ell-n$, extends uniquely to a well-defined $\Z$-algebra homomorphism $\eta: \Hc^*(\mathbb{G}_{n,\ell},\Z)\rightarrow B$.
\end{lem}

\begin{proof} By Theorem \ref{borel}, $\Hc^*(\mathbb{G}_{n,\ell},\Z)\cong \mathbb{Z}[x_1,\dots,x_{n},\bar{x}_1,\dots,\bar{x}_{\ell-n}]/I_{n,\ell-n}$. By Lemma \ref{commutation0}, to prove the lemma, it suffices to verify the generating relations given by $I_{n,\ell-n}$ for $b(\blam_i)$ and $b(\bmu_j)$.

Suppose the $\ell-n\geq n$, in this case, we should check
the following three types of relations,
\begin{align}
&b(\blam_m)+\sum_{s=1}^{m-1}(-1)^{m-s}b(\blam_s)b(\bmu_{m-s})+(-1)^{m}b(\bmu_{m})=0,\quad 1\leq m\leq n,\label{relation1}\\
&(-1)^mb(\bmu_m)+\sum_{s=1}^{n}(-1)^{m-s}b(\blam_s)b(\bmu_{m-s})=0,\quad n<m\leq\ell-n,\label{relation2}\\
&\sum_{\substack{i+j=m\\ 1\leq i\leq n\\ 1\leq j\leq \ell-n}}(-1)^jb(\blam_i)b(\bmu_j)=0,\quad \ell-n<m\leq\ell .\label{relation3}
\end{align}

Let $1\leq m\leq n$. We first verify the relation (\ref{relation1}). By Corollary \ref{the equivalence between zlam and Schur polynomial}, we have that
\begin{align*}
&\quad\,b(\blam_m)+\sum_{s=1}^{m-1}(-1)^{m-s}b(\blam_s)b(\bmu_{m-s})+(-1)^{m}b(\bmu_{m})\\
&=\psi_{w_{0}}(y_1\cdots y_m)y_{\min}\psi_{w_{0}}y_{\min}
+\sum_{s=1}^{m-1}(-1)^{m-s}\psi_{w_{0}}(y_1\cdots y_{s})y_{\min}\psi_{w_{0}}y_{\min}\psi_{w_{0}}y_1^{m-s}y_{\min}\psi_{w_{0}}y_{\min}\\
&\qquad +(-1)^{m}\psi_{w_{0}}y_1^my_{\min}\psi_{w_{0}}y_{\min}\\
&=(-1)^{\tfrac{n(n-1)}{2}}\Bigl(\psi_{w_{0}}e_m(y_1,\cdots,y_n)y_{\min}+\psi_{w_{0}}\sum_{s=1}^{m-1}(-1)^{m-s}e_s(y_1,\cdots,y_n)h_{m-s}(y_1,\cdots,y_n)y_{\min}\\
&\qquad +(-1)^{m}\psi_{w_{0}}h_m(y_1,\cdots,y_n)y_{\min}\Bigr)\\
&=(-1)^{\tfrac{n(n-1)}{2}}\psi_{w_{0}}\Bigl(\sum_{s=0}^{m}(-1)^{m-s}e_s(y_1,\cdots,y_n)h_{m-s}(y_1,\cdots,y_n)\Bigr)y_{\min}\\
&=0.
\end{align*}

Let $n<m\leq \ell-n$. We next verify the relation (\ref{relation2}). By Corollary \ref{the equivalence between zlam and Schur polynomial}, we have that
\begin{align*}
&\quad\,(-1)^mb(\bmu_m)+\sum_{s=1}^{n}(-1)^{m-s}b(\blam_s)b(\bmu_{m-s})\\
&=(-1)^m\psi_{w_{0}}y_1^{m}y_{\min}\psi_{w_{0}}y_{\min}
+\sum_{s=1}^{n}(-1)^{m-s}\psi_{w_{0}}(y_1\cdots y_{s})y_{\min}\psi_{w_{0}}y_{\min}\psi_{w_{0}}y_1^{m-s}y_{\min}\psi_{w_{0}}y_{\min}\\
&=(-1)^{\tfrac{n(n-1)}{2}}\Bigl((-1)^m\psi_{w_{0}}h_m(y_1,\cdots,y_n)y_{\min}+\psi_{w_{0}}\sum_{s=1}^{n}(-1)^{m-s}e_s(y_1,\cdots,y_n)h_{m-s}(y_1,\cdots,y_n)y_{\min}\Bigr)\\
&=(-1)^{\tfrac{n(n-1)}{2}}\psi_{w_{0}}\Bigl(\sum_{s=0}^{n}(-1)^{m-s}e_s(y_1,\cdots,y_n)h_{m-s}(y_1,\cdots,y_n)\Bigr)y_{\min}\\
&=0.
\end{align*}

Let $\ell-n<m\leq\ell$. We now verify the relation (\ref{relation3}). First note that for any $t\geq 1$, by Lemma \ref{sum1}, inside $\HH_{\ell,n}^{(0)}$ we have that
\begin{equation}\label{vanishiing} h_{\ell-n+t}(y_1,\cdots,y_n)=\sum_{l_1+\cdots+l_n=\ell-n+t}y_1^{l_1}\cdots y_n^{l_n}=0.\end{equation}
Therefore, using Corollary \ref{the equivalence between zlam and Schur polynomial} again, we have that
\begin{align*}
&\quad\,\sum_{\substack{i+j=m\\ 1\leq i\leq n,1\leq j\leq \ell-n}}(-1)^jb(\bmu_j)b(\blam_i)\\
&=\sum_{\substack{i+j=m\\ 1\leq i\leq n,1\leq j\leq \ell-n}}(-1)^{j}\psi_{w_{0}}y_1^jy_{\min}\psi_{w_{0}}y_{\min}\psi_{w_{0}}(y_1\cdots y_i)y_{\min}\psi_{w_{0}}y_{\min}\\
&=(-1)^{\tfrac{n(n-1)}{2}}\psi_{w_{0}}\sum_{\substack{i+j=m\\ 1\leq i\leq n,1\leq j\leq \ell-n}}(-1)^{j}h_j(y_1,\cdots,y_n)e_{i}(y_1,\cdots,y_n)y_{\min}\\
&=(-1)^{\tfrac{n(n-1)}{2}}\psi_{w_{0}}\Bigl(\sum_{s=\ell-n+1}^{m}(-1)^{s}e_{m-s}(y_1,\cdots,y_n)h_s(y_1,\cdots,y_n)\\
&\qquad\,\,+\sum_{\substack{i+j=m\\ 1\leq i\leq n,1\leq j\leq \ell-n}}(-1)^{j}e_i(y_1,\cdots,y_n)h_{j}(y_1,\cdots,y_n)\Bigr)y_{\min}\\
&=(-1)^{\tfrac{n(n-1)}{2}}\psi_{w_{0}}\Bigl(\sum_{s=0}^{n}(-1)^{m-s}e_{s}(y_1,\cdots,y_n)h_{m-s}(y_1,\cdots,y_n)\Bigr)y_{\min}\\
&=0,
\end{align*}
where we have used (\ref{vanishiing}) in the third equality and used Lemma \ref{alternating} in the last equality.

Suppose that $\ell-n<n$, we should check the following three types of relations,
\begin{align}
&b(\blam_m)+\sum_{s=1}^{m-1}(-1)^{m-s}b(\blam_s)b(\bmu_{m-s})+(-1)^{m}b(\bmu_{m})=0,\quad 1\leq m\leq \ell-n,\label{relation11}\\
&b(\blam_m)+\sum_{s=1}^{\ell-n}(-1)^{s}b(\bmu_{s})b(\blam_{m-s})=0,\quad \ell-n<m\leq n,\label{relation22}\\
&\sum_{\substack{i+j=m\\ 1\leq i\leq n\\ 1\leq j\leq \ell-n}}(-1)^jb(\bmu_j)b(\blam_i)=0,\quad n<m\leq\ell .\label{relation33}
\end{align}
The proof is similar and is left to the readers.
\end{proof}

\begin{dfn}(\cite{Man})\label{tensorpartition}
For any partition $\lambda$ and positive integer $k$, we define $\lambda\otimes k$ (respectively, $\lambda\otimes 1^k$) to be the set of all partitions obtained by adding $k$ boxes to $\lambda$ at most one per column (respectively, at most one per row).
\end{dfn}

The following is the Pieri's formulas for the elements $\{b_{\lambda}\mid\lam\in\P_n\}$.

\begin{prop}\label{pieri2}
For any partition $\lambda=(\lambda_1,\cdots,\lambda_n)\in\P_n$ and integers $0\leq k\leq n, s\geq0$, we have that
$$b_{\lambda}b_{(1^{k})}=\sum_{\mu\in\lambda\otimes1^k}b_{\mu},\qquad b_{\lambda}b_{(s)}=\sum_{\mu\in\lambda\otimes s}b_{\mu}.$$
\end{prop}
\begin{proof}
By Corollary \ref{the equivalence between zlam and Schur polynomial},
\begin{align*}
b_{\lambda}b_{(1^k)}&=\psi_{w_{0}}y_1^{\lambda_1+n-1}y_2^{\lambda_2+n-2}\cdots y_n^{\lambda_n}\psi_{w_{0}}y_{\min}\psi_{w_{0}}(y_1\cdots y_k)y_1^{n-1}y_2^{n-2}\cdots y_n\psi_{w_{0}}y_{\min}\\
&=(-1)^{\tfrac{n(n-1)}{2}}\psi_{w_{0}}y_1^{\lambda_1+n-1}y_2^{\lambda_2+n-2}\cdots y_n^{\lambda_n}\psi_{w_{0}}(y_1\cdots y_k)y_1^{n-1}y_2^{n-2}\cdots y_n\psi_{w_{0}}y_{\min}\\
&=\psi_{w_{0}}y_1^{\lambda_1+n-1}y_2^{\lambda_2+n-2}\cdots y_n^{\lambda_n}\sum_{1\leq t_1<\cdots<t_k\leq n}y_{t_1}\cdots y_{t_k}\psi_{w_{0}}y_{\min}.
\end{align*}
We just need to show the above last term is equal to $\sum_{\mu\in\lambda\otimes1^k}b_{\mu}.$

For each $k$-tuple $(t_1,\cdots,t_k)$ with $1\leq t_1<\cdots<t_k\leq n$, let $(a_1,\cdots,a_n)$ be the $n$-tuple which is determined by $$
a_i:=\begin{cases}\lam_i+1, &\text{if $i=t_s$ for some $1\leq s\leq k$;}\\
\lam_i, &\text{otherwise.}
\end{cases}
$$

If $(a_1,\cdots,a_n)$ is not a partition, then there must exist some $j$ with $1\leq j\leq k$ such that $$\lambda_{t_{j}-1}=\lambda_{t_j}\ \text{and}\ a_{t_j-1}=\lambda_{t_{j}-1},\ a_{t_j}=\lambda_{t_j}+1.$$
In this case, $a_{t_j-1}+1=a_{t_j}$. Let $w_j\in\Sym_n$ such that $w_{0}=s_{t_j-1}w_j$ and $\ell(w_{0})=\ell(w_j)+1$. Thus
$$\begin{aligned}
&\quad\,\psi_{w_0}y_1^{\lambda_1+n-1}y_2^{\lam_2+n-2}\cdots y_n^{\lam_n}(y_{t_1}\cdots y_{t_k})\psi_{w_{0}}y_{\min}\\
&=\psi_{w_0}y_1^{a_1+n-1}\cdots (y_{t_j-1}^{a_{t_j-1}+n-(t_j-1)}y_{t_j}^{a_{t_j}+n-t_j})\cdots y_n^{a_n}\psi_{w_{0}}y_{\min}\\
&=\psi_{w_0}y_1^{a_1+n-1}\cdots (y_{t_j-1}^{a_{t_j-1}+n-(t_j-1)}y_{t_j}^{a_{t_j}+n-t_j})\cdots y_n^{a_n}\psi_{t_j-1}\psi_{w_j}y_{\min}\\
&=\psi_{w_0}y_1^{a_1+n-1}\cdots (y_{t_j-1}^{a_{t_j-1}+n-(t_j-1)}y_{t_j}^{a_{t_j}+n-t_j}\psi_{t_j-1})\cdots y_n^{a_n}\psi_{w_j}y_{\min}\\
&=\psi_{w_0}y_1^{a_1+n-1}\cdots (\psi_{t_j-1}y_{t_j-1}^{a_{t_j-1}+n-(t_j-1)}y_{t_j}^{a_{t_j}+n-t_j})\cdots y_n^{a_n}\psi_{w_j}y_{\min}\\
&=\psi_{w_0}\psi_{t_j-1}y_1^{a_1+n-1}\cdots (y_{t_j-1}^{a_{t_j-1}+n-(t_j-1)}y_{t_j}^{a_{t_j}+n-t_j})\cdots y_n^{a_n}\psi_{w_j}y_{\min}\\
&=0 ,
\end{aligned}
$$ as required. This proves the first formula.

For the second formula, by Corollary \ref{the equivalence between zlam and Schur polynomial}, we have that
\begin{align*}
b_{\lambda}b_{(s)}&=\psi_{w_{0}}y_1^{\lambda_1+n-1}y_2^{\lambda_2+n-2}\cdots y_n^{\lambda_n}\psi_{w_{0}}y_{\min}\psi_{w_{0}}y_1^sy_1^{n-1}y_2^{n-2}\cdots y_n\psi_{w_{0}}y_{\min}\\
&=\psi_{w_{0}}y_1^{\lambda_1+n-1}y_2^{\lambda_2+n-2}\cdots y_n^{\lambda_n}\sum_{l_1+\cdots+l_n=s}y_1^{l_1}\cdots y_n^{l_n}\psi_{w_{0}}y_{\min}\\
\end{align*}
Hence to prove the second formula, we just need to show the above last term
is equal to $\sum_{\mu\in\lambda\otimes s}b_{\mu}$.

Let $\underline{c}=(c_1,\cdots,c_n)$ be any given $n$-tuple. Suppose that $c_j\neq c_{j+1}$ for some $1\leq j<n$. Let $u_j\in\Sym_n$ such that
$w_{0}=s_{j}u_j$ and $\ell(w_{0})=\ell(u_j)+1$.  Then in this case $$
\begin{aligned}
&\quad\,\psi_{w_0}y_1^{c_1}\cdots (y_j^{c_j}y_{j+1}^{c_{j+1}}+y_j^{c_{j+1}}y_{j+1}^{c_j})\cdots y_n^{c_n}\psi_{w_{0}}y_{\min}\\
&=\psi_{w_0}y_1^{c_1}\cdots (y_j^{c_j}y_{j+1}^{c_{j+1}}+y_j^{c_{j+1}}y_{j+1}^{c_j})\psi_j\cdots y_n^{c_n}\psi_j\psi_{u_j}y_{\min}\\
&=\psi_{w_0}y_1^{c_1}\cdots \psi_j(y_j^{c_j}y_{j+1}^{c_{j+1}}+y_j^{c_{j+1}}y_{j+1}^{c_j})\cdots y_n^{c_n}\psi_{u_j}y_{\min}\\
&=\psi_{w_0}\psi_jy_1^{c_1}\cdots (y_j^{c_j}y_{j+1}^{c_{j+1}}+y_j^{c_{j+1}}y_{j+1}^{c_j})\cdots y_n^{c_n}\psi_{u_j}y_{\min}\\
&=0.
\end{aligned} $$
In other words, in this case,
\begin{equation}\label{neqCase}
\psi_{w_0}y_1^{c_1}\cdots y_j^{c_j}y_{j+1}^{c_{j+1}}\cdots y_n^{c_n}\psi_{w_{0}}y_{\min}
=-\psi_{w_0}y_1^{c_1}\cdots y_j^{c_{j+1}}y_{j+1}^{c_{j}}\cdots y_n^{c_n}\psi_{w_{0}}y_{\min}.
\end{equation}

If $c_j=c_{j+1}$  for some $1\leq j<n$, then a similar calculation shows that \begin{equation}\label{eqCase}
\psi_{w_0}y_1^{c_1}\cdots y_j^{c_j}y_{j+1}^{c_{j+1}}\cdots y_n^{c_n}\psi_{w_{0}}y_{\min}=0.
\end{equation}

Now let $\alpha=(\alpha_1,\cdots,\alpha_n)\in\N^n$ with $\alpha_1+\cdots+\alpha_n=s$. If $\lambda+\alpha:=(\lam_1+\alpha_1,\cdots,\lam_n+\alpha_n)$ doesn't belong to $\lambda\otimes s$, then by the definition of $\lambda\otimes s$, there must exists some $1\leq j< n$
such that $\lambda_{j+1}+\alpha_{j+1}>\lambda_j$. In this case, we define an $n$-tuple $\beta=(\beta_1,\cdots,\beta_n)\in\N^n$ by
$$
\beta_{j}=\alpha_{j+1}-(\lambda_j-\lambda_{j+1}+1),\quad\beta_{j+1}=\alpha_{j}+(\lambda_j-\lambda_{j+1}+1),\quad\beta_i=\alpha_i,\quad \forall\,i\neq j, j+1.
$$

Set $\delta=(n-1,n-2,\cdots,0)$. Note that $\beta_1+\cdots+\beta_n=s$ and $s_j(\lambda+\alpha+\delta)=\lambda+\beta+\delta$. If $\lambda+\alpha+\delta=\lambda+\beta+\delta$, then (\ref{eqCase})
implies that $\psi_{w_0}y_1^{\lambda_1+\alpha_1+n-1}\cdots y_{n}^{\lambda_n+\alpha_n}\psi_{w_{0}}y_{\min}=0$;
if $\lambda+\alpha+\delta\neq \lambda+\beta+\delta$, then (\ref{neqCase}) implies that $$
\psi_{w_0}y_1^{\lambda_1+\alpha_1+n-1}\cdots y_{n}^{\lambda_n+\alpha_n}\psi_{w_{0}}y_{\min}
+\psi_{w_0}y_1^{\lambda_1+\beta_1+n-1}\cdots y_{n}^{\lambda_n+\beta_n}\psi_{w_{0}}y_{\min}=0.
$$
Therefore, in any case, we can ignore these terms when calculating
$$\psi_{w_{0}}y_1^{\lambda_1+n-1}y_2^{\lambda_2+n-2}\cdots y_n^{\lambda_n}\sum_{l_1+\cdots+l_n=s}y_1^{l_1}\cdots y_n^{l_n}\psi_{w_{0}}y_{\min}.$$
Thus it suffices to consider only those $(l_1,\cdots,l_n)\in\N^n$ with $l_1+\cdots+l_n=s$ and $\lambda+(l_1,\cdots,l_n)\in\lambda\otimes s$. This complete the proof of the second formula.
\end{proof}

The following is the Jacobi-Trudi formula for the elements $\{b_{\lambda}\mid\lam\in\P_n\}$.

\begin{lem}\label{Jacobi3}
Let $\lambda=(\lambda_1,\cdots,\lambda_n)\in\P_n$. Let $\lambda^{\prime}:=(\lambda^{\prime}_1,\dots,\lambda^{\prime}_m)$ be the conjugate of $\lam$, where $m\geq\ell(\lambda^{\prime})$. Then we have that,
$$b_{\lambda}= \det(b_{(\lambda_i-i+j)})_{1\leq i,j\leq n},\ \ \ \ b_{\lambda}=\det(b_{(1^{\lambda^{\prime}_i-i+j})})_{1\leq i,j\leq m}.$$
where, by definition, $b_{(s)}=0,b_{(1^s)}=0$ if $s<0$ and $b_{(1^k)}=0$ if $k>n$.
\end{lem}

\begin{proof}
First, by Corollary \ref{the equivalence between zlam and Schur polynomial}, note that
$$b_{(0)}=b_{(1^0)}=\psi_{w_{0}}y_1^{n-1}\cdots y_{n-1}\psi_{w_{0}}y_1^{n-1}\cdots y_n\equiv 1\mod J_n.$$
By a simple computation, one can easily observe that the determinants in the theorem remain unchanged if we restrict the order of the determinants to the corresponding length of the partitions. The observation allows us to use induction on the length of the partition (or its conjugate) to prove the formulas.

For the first formula, we use induction on the length $l:=\ell(\lam)$ of $\lambda$.
Expanding the determinant along the last column, we have following
alternating sum,
$$\sum_{i=1}^{l}(-1)^{l-i}b_{(\lambda_1,\cdots,\lambda_{i-1},\lambda_{i+1}-1,\cdots,\lambda_l-1)}b_{(\lambda_i-i+l)}.$$

For each $1\leq t\leq l+1$, we denote by $B_t$ the set of all partitions $\mu=(\mu_1,\cdots,\mu_l)$ with $|\mu|=|\lambda|$ and
such that $\lambda_j\leq \mu_j\leq \lambda_{j-1}$ for $j<t$,
and $\lambda_{j+1}-1\leq \mu_j\leq \lambda_j-1$ for $j\geq t$, where $\lambda_0:=+\infty, \lambda_{l+1}:=0$.
By definition, $B_1=\emptyset$ and $B_{l+1}=\{\lambda\}$.

For each $1\leq i\leq l$, by definition, we have a decomposition:
$$(\lambda_1,\cdots,\lambda_{i-1},\lambda_{i+1}-1,\cdots,\lambda_l-1)\otimes(\lambda_i+l-i)=B_i\sqcup B_{i+1}.$$
Using Pieri's Formula (the second formula in Proposition \ref{pieri2}), we may write the $i$th term
$$b_{(\lambda_1,\cdots,\lambda_{i-1},\lambda_{i+1}-1,\cdots,\lambda_l-1)}b_{(\lambda_i-i+l)}$$ in following form,
$$\sum_{\mu\in B_i}b_{\mu}+\sum_{\mu\in B_{i+1}}b_{\mu}$$
Therefore, after cancelations, we have the first formula.

For the second formula, we use induction on the length $l^{\prime}$ of the conjugate partition $\lambda^{\prime}=(\lambda^{\prime}_1,\dots,\lambda^{\prime}_{l^{\prime}})$.
Expanding the determinant along the last column, we have following alternating sum,
$$\sum_{i=1}^{l^{\prime}}(-1)^{l^{\prime}-i}b_{\nu_i}b_{1^{\lambda^{\prime}_i-i+l^{\prime}}}$$
where $\nu_i:=(\lambda^{\prime}_1,\cdots,\lambda^{\prime}_{i-1},\lambda^{\prime}_{i+1}-1,
\cdots,\lambda^{\prime}_{l^{\prime}}-1)^{\prime}$.

Note that
$$\nu_i\otimes1^{\lambda^{\prime}_i-i+l^{\prime}}=((\nu_i)^{\prime}\otimes(\lambda^{\prime}_i-i+l^{\prime}))^{\prime}$$

For each $1\leq t\leq l^{\prime}+1$, we denote by $C_t$ the set of all partitions $\mu=(\mu_1,\cdots,\mu_{l^{\prime}})$ with
$|\mu|=|\lambda^{\prime}|$ and such that $\lambda^{\prime}_j\leq \mu_j\leq \lambda^{\prime}_{j-1}$ for $j<t$,
and $\lambda^{\prime}_{j+1}-1\leq \mu_j\leq \lambda^{\prime}_j-1$ for $j\geq t$, where $\lambda^{\prime}_0:=+\infty, \lambda^{\prime}_{l^{\prime}+1}:=0$.
By definition, $C_1=\emptyset$ and $C_{l^{\prime}+1}=\{\lambda^{\prime}\}$.

For each $1\leq i\leq l^{\prime}$, by definition, we have decomposition $(\nu_i)^{\prime}\otimes(\lambda'_i+l'-i)=C_i\sqcup C_{i+1}$, hence
$$\nu_i\otimes1^{\lambda^{\prime}_i-i+l^{\prime}}=(C_i)^{\prime}\sqcup (C_{i+1})^{\prime}.$$
Using Pieri's Formula (the first formula in Proposition \ref{pieri2}),
we may write the $i$th term $b(\nu_i)b(1^{\lambda^{\prime}_i-i+l^{\prime}})$
in following form,
$$\sum_{\mu\in (C_i)^{\prime}}b_{\mu}+\sum_{\mu\in (C_{i+1})^{\prime}}b_{\mu}.$$
Therefore, after cancelations, we have the second formula.
\end{proof}

\begin{lem}\text{(Pieri formula, \cite[(3.6)]{Hi})} For each $(a_1,\cdots,a_{n})\in\Theta_{\ell,n}$ and $0\leq j\leq \ell-n$, $$
(a_1,\cdots,a_{n})\tilde{c}_j=\sum_{\substack{a_i\leq b_i\leq a_{i+1}\\ \sum b_i=j+\sum a_i}}(b_1,\cdots,b_{n}) .
$$
\end{lem}

\begin{lem}\text{(Giambelli formula, \cite[(3.7)]{Hi})} \label{Giam} For each $(a_1,\cdots,a_{n})\in\Theta_{\ell,n}$, $$
(a_1,\cdots,a_{n})=\det\Bigl(\tilde{c}_{a_i+i-j}\Bigr)_{1\leq i,j\leq n} ,
$$
where by convention $\tilde{c}_t:=0$ if $t$ is not between $0$ and $\ell-n$.
\end{lem}

\begin{rem} We remark that there exist some obvious typos in \cite[(3.6), (3.7)]{Hi}. That is, the elements $\bar{c}_j$ and $\bar{c}_{a_i+j-i}$ in \cite[(3.6), (3.7)]{Hi} should be replaced by $(-1)^j\bar{c}_j$ and $(-1)^{a_i+i-j}\bar{c}_{a_i+i-j}$ respectively, see the proof of (3.7) in \cite[Page 117]{Hi}.
\end{rem}

The following proposition is the main result of this paper.

\begin{prop}\label{mainthprop2}
The $\Z$-algebra homomorphism $\eta$ constructed in Lemma \ref{homomorphism} is an isomorphism. Moreover, for any $(a_1,\cdots,a_{n})\in\Theta_{\ell,n}$, $$
\eta\bigl((a_1,\cdots,a_{n})\bigr)=b\bigl(\tau(a_1,\cdots,a_{n})\bigr),
$$
where $\tau$ is defined in Definition \ref{bijective}.
\end{prop}
\begin{proof} By Lemma \ref{basis2} and Lemma \ref{homomorphism}, it suffices to prove the second part of the proposition.

By (\ref{taudfn}), $$\tau(a_1,\cdots,a_{n})=(\ell+1-(a_n+n),\cdots,\ell+1-(a_2+2),\ell+1-(a_1+1)).$$ Let $\lam=(\lam_1,\cdots,\lam_n):=\rho\circ\tau(a_1,\cdots,a_{n})$. Then $\lam_i=a_{n+1-i},\,\forall\,1\leq i\leq n$.

Now, using Lemma \ref{Jacobi3} and Lemma \ref{Giam}, we get that $$\begin{aligned}
b\bigl(\tau(a_1,\cdots,a_{n})\bigr)&=\pi_{\ell}\bigl(b_{\rho\circ\tau(a_1,\dots,a_n)}\bigr)\\
&=\det\Bigl(\pi_{\ell}\bigl(b_{(\lam_i-i+j)}\bigr)\Bigr)_{1\leq i,j\leq n}\\
&=\det\Bigl(b(\bmu_{\lam_i-i+j})\Bigr)_{1\leq i,j\leq n}\\
&=\det\Bigl(b(\bmu_{a_{n+1-i}-i+j})\Bigr)_{1\leq i,j\leq n}\\
&=\det\Bigl(b(\bmu_{a_{i}+i-j})\Bigr)_{1\leq i,j\leq n}\\
&=\det\Bigl(\eta\bigl(\tilde{c}_{a_{i}+i-j}\bigr)\Bigr)_{1\leq i,j\leq n}\\
&=\eta\Bigl(\det\bigl(\tilde{c}_{a_i+i-j}\bigr)_{1\leq i,j\leq n}\Bigr)=\eta\bigl((a_1,\cdots,a_{n})\bigr), \end{aligned}
$$
as required. This completes the proof of the proposition.
\end{proof}

\bigskip

\section{The duality}

In this section, we shall first show that there is a similar $\Z$-algebra isomorphism between the cohomology algebra $\Hc^*(\mathbb{G}_{\ell-n,\ell},\mathbb{Z})$
and the natural $\Z$-form $B$ of the basic algebra of $\HH_{\ell,n}^{(0)}$. Then we shall use it to explicitly describe the image of each Schubert class basis element, which is again another Schubert class basis element, under certain natural isomorphism between $\Hc^*(\mathbb{G}_{n,\ell},\mathbb{Z})$ and $\Hc^*(\mathbb{G}_{\ell-n,\ell},\mathbb{Z})$.

By definition,  $$
\Theta_{\ell,\ell-n}:=\bigl\{(a_1,\cdots,a_{\ell-n})\bigm|0\leq a_1\leq\cdots\leq a_{\ell-n}\leq n, a_i\in\Z,\forall\,i\bigr\} .
$$

\begin{dfn}\label{bijective2} Let $\hat{\tau}$ be the bijection between $\Theta_{\ell,\ell-n}$ and $\P_{\ell,n}$ which is defined as follows:
for any $(\ell-n)$-tuple $\underline{a}=(a_1,\dots,a_{\ell-n})\in\Theta_{\ell,\ell-n}$, $\hat{\tau}{(\underline{a})}$ is the unique $n$-tuple $\blam=(k_1,\dots,k_n)\in\P_{\ell,n}$
such that $\blam$ is obtained from deleting the $(\ell-n)$-tuple $(a_1+1,\cdots,a_{\ell-n}+\ell-n)$ inside the $\ell$-tuple $(1,2,\cdots,\ell)$.
\end{dfn}

\begin{thm}\label{mainthm2}
Let $B$ be the natural $\Z$-form of the $\Z$-graded basic algebra of $\HH_{\ell,n}^{(0)}$.
The map $\hat{\eta}$ which sends $c_i=x_i+I_{\ell-n,n}$ to $b(\bmu_i)$ for each $1\leq i\leq \ell-n$, and
$\bar{c}_j=\bar{x}_j+I_{\ell-n,n}$ to $(-1)^jb(\blam_j)$ for each $1\leq j\leq n$,
extends uniquely to a well-defined $\Z$-algebra isomorphism $\hat{\eta}: \Hc^*(\mathbb{G}_{\ell-n,\ell},\Z)\rightarrow B$.
Moreover, for any $(a_1,\cdots,a_{\ell-n})\in\Theta_{\ell,\ell-n}$,
$$\hat{\eta}\bigl((a_1,\cdots,a_{\ell-n})\bigr)=b\bigl(\hat{\tau}(a_1,\cdots,a_{\ell-n})\bigr).$$
\end{thm}

\begin{proof} Mimicking the proof of Lemma \ref{homomorphism}, one can show that $\hat{\eta}$ is a well-defined $\Z$-algebra homomorphism. It suffices to prove the second part of the theorem.

We first claim that \begin{equation}\label{conjugate1}
\bigl(\rho\circ\hat{\tau}(a_1,\cdots,a_{\ell-n})\bigr)'=(a_{\ell-n},a_{\ell-n-1},\cdots,a_2,a_1) .
\end{equation}

In fact, write $\blam=(k_1,\cdots,k_n):=\hat{\tau}(a_1,\cdots,a_{\ell-n})$. By a direct verification, we can see that $$
(a_1,\cdots,a_{\ell-n})=\bigl(\underbrace{0,\cdots,0,}_{\text{$k_1-1$ copies}}\underbrace{1,\cdots,1,}_{\text{$k_2-k_1-1$ copies}}\cdots,
\underbrace{n-1,\cdots,n-1,}_{\text{$k_n-k_{n-1}-1$ copies}}\underbrace{n,\cdots,n}_{\text{$\ell-k_n$ copies}}\bigr).
$$
In other words, the inverse $\hat{\tau}^{-1}$ of $\hat{\tau}$ is given by $$
\hat{\tau}^{-1}(\blam)=\bigl(\underbrace{0,\cdots,0,}_{\text{$k_1-1$ copies}}\underbrace{1,\cdots,1,}_{\text{$k_2-k_1-1$ copies}}\cdots,
\underbrace{n-1,\cdots,n-1,}_{\text{$k_n-k_{n-1}-1$ copies}}\underbrace{n,\cdots,n}_{\text{$\ell-k_n$ copies}}\bigr)\in\Theta_{\ell,\ell-n}.
$$

By definition, $$
\rho(\blam)=\bigl(\ell-k_1-(n-1),\ell-k_2-(n-2),\cdots,\ell-k_{n-1}-1,\ell-k_n\bigr) .
$$
For any integer $k\geq 1$, we have that $\rho(\blam)'_{k}=\#\{t\geq 1\mid\rho(\blam)_t\geq k\}$. Let $s$ be the maximal integer such that $$
\ell-k_s-(n-s)\geq k,\quad \ell-k_t-(n-t)\leq k-1,\qquad\forall\,s+1\leq t\leq n .
$$
In particular, this implies that $k_{s+1}-k_s-1>0$. Since $$
\#\{i\mid a_i=(\hat{\tau}^{-1}(\blam))_i>s\}=(k_{s+2}-k_{s+1}-1)+(k_{s+3}-k_{s+2}-1)+\cdots+(\ell-k_n)=\ell-k_{s+1}-(n-s-1)\leq k-1,
$$
it follows that $a_{\ell-n-(k-1)}\leq s$. On the other hand, $a_{\ell-n-(k-1)}<s$ can not happen because $$
\#\{i|a_i=(\hat{\tau}^{-1}(\blam))_i>s-1\}=(k_{s+1}-k_{s}-1)+(k_{s+2}-k_{s+1}-1)+\cdots+(\ell-k_n)=\ell-k_{s}-(n-s)\geq k .
$$
This proves that $a_{\ell-n-(k-1)}=s$ as required. It is clear $s=\rho(\blam)'_k$.
Hence $\rho(\blam)^{\prime}_k=a_{\ell-n+1-k}$. This completes the proof of claim (\ref{conjugate1}).

Let $\mu=(\mu_1,\cdots,\mu_n):=\rho\circ\tau(a_1,\cdots,a_{\ell-n})$. Now, using Lemma \ref{Jacobi3} and Lemma \ref{Giam},
we get that
$$\begin{aligned}
b\bigl(\hat{\tau}(a_1,\cdots,a_{\ell-n})\bigr)&=\pi_{\ell}\bigl(b_{\rho\circ\hat{\tau}(a_1,\cdots,a_{\ell-n})}\bigr)\\
&=\det\Bigl(\pi_{\ell}\bigl(b_{(1^{\mu^{\prime}_i-i+j})}\bigr)\Bigr)_{1\leq i,j\leq \ell-n}\\
&=\det\Bigl(b(\blam_{\mu^{\prime}_i-i+j})\Bigr)_{1\leq i,j\leq \ell-n}\\
&=\det\Bigl(b(\blam_{a_{\ell-n+1-i}-i+j})\Bigr)_{1\leq i,j\leq \ell-n}\\
&=\det\Bigl(b(\blam_{a_{i}+i-j})\Bigr)_{1\leq i,j\leq \ell-n}\\
&=\det\Bigl(\hat{\eta}\bigl(\tilde{c}_{a_{i}+i-j}\bigr)\Bigr)_{1\leq i,j\leq \ell-n}\\
&=\hat{\eta}\Bigl(\det\bigl(\tilde{c}_{a_i+i-j}\bigr)_{1\leq i,j\leq \ell-n}\Bigr)=\hat{\eta}\bigl((a_1,\cdots,a_{\ell-n})\bigr),
\end{aligned}$$
as required. This completes the proof of the proposition.
\end{proof}

\begin{rem} We remark that Xingyu Dai has considered the inverse map $\hat{\tau}^{-1}$ and claimed the existence of a similar isomorphism in \cite[Theorem 3.1]{Dai}. However, there are several serious gaps in his ``proof". Not only did he mix the Grassmannian $\mathbb{G}_{n,\ell}$ with the Grassmannian $\mathbb{G}_{\ell-n,\ell}$ in \cite{Dai} but also the key results (\cite[Lemma 3.8, Proposition 3.9]{Dai}) on which his ``proof" of  \cite[Theorem 3.1]{Dai} relies are both false.
\end{rem}

Recall that there is a natural isomorphism $\zeta$ between the cohomology algebra $\Hc^*(\mathbb{G}_{n,\ell},\Z)=\Z\<c_1,\cdots,c_n,\bar{c}_1,\cdots,\bar{c}_{\ell-n}\>$ and the cohomology algebra $\Hc^*(\mathbb{G}_{\ell-n,\ell},\Z)=\Z\<c_1,\cdots,c_{\ell-n},\bar{c}_1,\cdots,\bar{c}_{n}\>$ which is defined on generators by $$
\zeta(c_i):=(-1)^i\bar{c}_i,\quad \zeta(\bar{c}_j):=(-1)^jc_j,\qquad \forall\,1\leq i\leq n,\, 1\leq j\leq \ell-n .
$$

\begin{thm} \label{mainthm3} With the notations as above, we have that, for any $(a_1,\cdots,a_n)\in\Theta_{\ell,n}$, $\zeta\bigl((a_1,\cdots,a_n)\bigr)$ is equal to the Schubert class basis element represented by the $(\ell-n)$-tuple $\hat{\tau}^{-1}\circ\tau(a_1,\cdots,a_n)$.
\end{thm}

\begin{proof} This follows easily from Proposition \ref{mainthprop2} and Theorem \ref{mainthm2} and the fact that $\zeta=\hat{\eta}^{-1}\circ\eta$, which can be checked directly on each generators $c_i, \bar{c}_j$, $i=1,2,\cdots,n, j=1,2,\cdots,\ell-n$.
\end{proof}

\begin{cor} The graded basic algebra of $\HH_{\ell,n}^{(0)}$ is isomorphic to the graded basic algebra of $\HH_{\ell,\ell-n}^{(0)}$. In particular, the cyclotomic nilHecke algebra $\HH_{\ell,n}^{(0)}$ is graded Morita equivalent to the cyclotomic nilHecke algebra $\HH_{\ell,\ell-n}^{(0)}$.
\end{cor}

Finally, we give the following corollary which can be regarded as a second version of Giambelli formula for Schubert classes (compare \cite[(3.7)]{Hi}).
It seems to be new as we did not find it anywhere in the references.

\begin{cor} \label{lastcor} For each $(a_1,\cdots,a_{n})\in\Theta_{\ell,n}$, inside $\Hc^*\bigl(\mathbb{G}_{\ell-n,\ell},\Z\bigr)$ we have that $$
\hat{\tau}^{-1}\circ\tau\bigl((a_1,\cdots,a_{n})\bigr)=\det\Bigl({c}_{a_i+i-j}\Bigr)_{1\leq i,j\leq n} ,
$$
where by convention ${c}_t:=0$ if $t$ is not between $0$ and $\ell-n$.
\end{cor}

\begin{proof} This follows from Proposition \ref{mainthprop2},  Theorem \ref{mainthm2}, Theorem \ref{mainthm3} and  \cite[(3.7)]{Hi}.
\end{proof}

\end{document}